\documentclass[reqno]{amsart}
\usepackage{amsmath,amssymb,amsthm, verbatim}
\usepackage{url}
\usepackage[usenames, dvipsnames]{color}

\usepackage[letterpaper,hmargin=1in,vmargin=1in]{geometry}

\usepackage{graphicx}
\usepackage{enumerate,pinlabel}
\usepackage{mathrsfs,graphicx,url}
\usepackage[usenames,dvipsnames]{color}
\usepackage[colorlinks=true,linkcolor=Black,citecolor=Black, urlcolor=Black]{hyperref}

\numberwithin{equation}{section}

\theoremstyle{plain}

\newtheorem*{theorem*}{Theorem}
\newtheorem*{lemma*}{Lemma}
\newtheorem{lemma}{Lemma}

\theoremstyle{definition}

\theoremstyle{remark}

\newcommand\supp{\mathop{\rm supp}}
\newcommand\real{\mathop{\rm Re}}
\newcommand\imag{\mathop{\rm Im}}
\newcommand*{\defeq}{\mathrel{\vcenter{\baselineskip0.5ex \lineskiplimit0pt
                     \hbox{\scriptsize.}\hbox{\scriptsize.}}}%
                     =}

\author{Jacob Shapiro}
\title{Semiclassical resolvent bound for compactly supported $L^\infty$ potentials}

\begin{document}
\begin{abstract}
We give an elementary proof of a weighted resolvent estimate  for
semiclassical Schr\"odinger operators in dimension $n \ge 1$. We require the potential belong to $L^\infty(\mathbb{R}^n)$ and have compact support, but do not require that it have distributional derivatives in $L^\infty(\mathbb{R}^n)$. The weighted resolvent norm is bounded by $e^{Ch^{-4/3}\log(h^{-1})}$, where $h$ is the semiclassical parameter.
\end{abstract}
\maketitle 
\author
\section{Introduction}
Let $\Delta \le 0$ be the Laplacian on $\mathbb{R}^n$, $n \ge 1$. We consider semiclassical Schr\"odinger operators of the form
\begin{equation*}
P = P(h) \defeq -h^2 \Delta + V : L^2(\mathbb{R}^n) \to L^2(\mathbb{R}^n),\qquad h > 0,
\end{equation*}
where the potential $V \in L^\infty(\mathbb{R}^n)$ is real-valued and compactly supported. By the Kato-Rellich Theorem, the operator $P$ is self-adjoint with respect to the domain $H^2(\mathbb{R}^n)$. Therefore, the resolvent $(P - z)^{-1}$ is bounded $L^2(\mathbb{R}^n) \to L^2(\mathbb{R}^n)$ for all $z \in \mathbb{C} \setminus \mathbb{R}$. We establish an $h$-dependent bound on a weighted resolvent norm that is uniform up to the positive real spectrum.

\begin{theorem*} \label{bounded potentials}
Let $n \ge 1$, $V \in L^\infty_{\text{\emph{comp}}}(\mathbb{R}^n)$, and $[E_{\text{\emph{min}}}, E_{\text{\emph{max}}}] \subseteq (0, \infty)$. For any $s > 1/2$, there exist $C, h_0>0$  such that
\begin{equation}\label{exp bound} 
\left\| (1+|x|)^{-s} (P(h)  -i\varepsilon)^{-1} (1+|x|)^{-s}  \right\|_{L^2(\mathbb{R}^n) \to H^2(\mathbb{R}^n)}  \le e^{Ch^{-4/3}\log(h^{-1})}, 
\end{equation}
 for all $E \in [E_{\text{\emph{min}}}, E_{\text{\emph{max}}}]$, $0 < \varepsilon < 1 $, and $h \in (0,h_0]$.
\end{theorem*}

Exponential resolvent bounds are known to hold under a wide range of  geometric, regularity, and decay assumptions. In \cite{bu98}, Burq showed the resolvent is $O(e^{Ch^{-1}})$ for smooth, compactly supported perturbations of the Laplacian outside an obstacle. He later established the same bound for smooth, long-range perturbations \cite{bu02}. Cardoso and Vodev \cite{cavo} extended Burq's estimate in \cite{bu02} to infinite volume Riemannian manifolds which may contain cusps. 

In lower regularity, Datchev \cite{da14} and the author \cite{sh16} proved the weighted resolvent norm in \eqref{exp bound} is still $O(e^{Ch^{-1}})$, provided $V$ and $\nabla V$ belong to $L^\infty$ and have long-range decay. Vodev \cite{vod14} showed an $O(e^{Ch^{-\ell}})$ bound, $0 < \ell < 1$, for potentials that are H\"older continuous, $h$-dependent, and have decay depending on $\ell$. 

Since the completion of the first draft of this paper, the author has learned about the independent and parallel work of Klopp and Vogel \cite{klvo18}. They use a different Carleman estimate to show that, if the support of $V$ is contained in the ball $B(0,R) \defeq \{x \in \mathbb{R}^n : |x| < R \}$, and $\chi$ is a smooth cutoff supported near $B(0,R)$, then for any compact interval $I \subseteq \mathbb{R} \setminus \{0\}$, there exist constants $C>0$ and $h_0 \in (0,1]$ such that
\begin{equation} \label{Klopp-Vogel}
\|(-h^2 \Delta + V - \lambda^2)^{-1} v\|_{H^1(B(0,R))} \le e^{Ch^{-4/3}\log(h^{-1})}\|v\|_{L^2(B(0,R))},
\end{equation}
 for all $ h \in (0,h_0]$, $v \in L^2_{\text{comp}}(B(0,R))$, and $\lambda \in I$. 

The novety of \eqref{exp bound} and \eqref{Klopp-Vogel} is that they are the first explicit $h$-dependent weighted resolvent bounds in such low regularity. 
The only previous result for general $L^\infty$ potentials of which the author is aware is by Rodnianski and Tao \cite[Theorem 1.7]{rota15}. They consider short-range, $L^\infty$ potentials on asymptotically conic manifolds of dimension $n \ge 3$, and prove a non-semiclassical version of \eqref{exp bound} in which the right side is replaced by an unspecified function of $h$.

A related result for compactly supported $L^\infty$ potentials in dimension one is \cite[Theorem 2.29]{dyzw}. It says that, given $V \in L^\infty_{\text{comp}}(\mathbb{R})$ and $[E_{\text{min}}, E_{\text{max}}] \subseteq (0,\infty)$, there exists a constant $c > 0$ such that the meromorphic continuation of the cutoff resolvent 
\begin{equation*} \label{semiclassical cutoff resolv}
\chi(-h^2 \partial_x + V - z)^{-1} \chi, \qquad \chi \in C^\infty_0(\mathbb{R}),
\end{equation*}
from $\imag z > 0$, $\real z > 0$ to $\imag z \le 0$, $\real z > 0$  has the property 
\begin{equation} \label{dim one exp region}
\text{$z$ is a pole, or \textit{resonance}, of the continuation, $\real z \in [E_{\text{min}}, E_{\text{max}}] \implies \imag z \ge - e^{-ch^{-1}}$}.
\end{equation}
See section \cite[Section 2.8]{dyzw} for further details. Resonance free regions are closely related to resolvent bounds, see, for instance,\cite[Theorem 1.5]{vod14}, and \cite[Theorem 2.8]{dyzw}, \cite[Theorem 3]{klvo18}, so \eqref{dim one exp region} strongly suggests that, in dimension one, the right side of \eqref{exp bound} can be improved to $e^{Ch^{-1}}$.

The $O(e^{Ch^{-1}})$ bound appearing in \cite{bu98, bu02, da14,sh16, cavo} is well-known to be optimal. See for instance, \cite{ddz15} and the references cited there. However, it is still an open problem to determine the optimal resolvent bound for $V \in L^\infty$.

Resolvent bounds such as \eqref{exp bound} imply local energy decay for the wave equation
\begin{equation} \label{wave equation}
\begin{cases}
(\partial_t^2 - c^2(x)\Delta) u(x,t) = 0, & (x,t) \in \left(\mathbb{R}^n \setminus \Omega \right) \times (0, \infty), \\
 u(x,0) = u_0,\\
 \partial_t u(x,0) = u_1(x), \\
 u(t,x) = 0, & (x,t) \in  \partial \Omega \times (0,\infty),
\end{cases}
\end{equation}
where $\Omega$ is a compact (possibly empty) obstacle with smooth boundary. 

Burq used his $O(e^{Ch^{-1}})$ resolvent bounds to show that, if $c$ is smooth and decaying sufficiently quickly to unity at infinity, then for any compact $K \subseteq \mathbb{R}^n \setminus \Omega$ and any compactly supported initial data, the local energy $E_K(t)$ of the solution to \eqref{wave equation} decays like
\begin{equation} \label{local energy brief}
E_K(t) \le \frac{C_{K, u_0, u_1}}{\log(2 + t)},  \qquad t \ge 0.
\end{equation}
See \cite[Theorem 1]{bu98} and \cite[Theorem 2]{bu02}. 

Logarithmic decays also hold in lower regularity. In \cite{sh17}, the author showed that \eqref{local energy brief} holds if $\Omega = \emptyset$, $n \ge 2$, and $c$ is Lipschitz perturbation of unity.  

Adapting the methods from these articles, one expects to find that if $c$ is a sufficiently decaying $L^\infty$-perturbation of unity, then
\begin{gather}
\text{$O(e^{Ch^{-\ell}})$ resolvent bound, $\ell \ge 1 $, implies} \nonumber \\ E_K(t) \le \frac{C_{K, u_0, u_1}}{\log^{1/\ell}(2 + t)}, \qquad t \ge 0. \label{infinity wave decay}
\end{gather}
In fact, the author has been informed by Georgi Vodev that \eqref{infinity wave decay} also follows by adapting the methods of \cite{cavo04} in a straightforward manner.  

Stronger resolvent bounds are known when $V$ is more regular and additional assumptions are made about the Hamilton flow $\Phi(t) = \text{exp}t(2 \xi \cdot \nabla_x - \nabla_x V \cdot \nabla_{\xi}) $. For example, if $V \in C^\infty_0(\mathbb{R}^n)$ and is  \textit{nontrapping} at the energy $E$, then it is well-known that \eqref{exp bound} improves to
\begin{equation*}
\left\|(1+|x|)^{-s} (P(h) - E -i\varepsilon)^{-1} (1+|x|)^{-s} \right\|_{L^2(\mathbb R^n) \to L^2(\mathbb R^n)} \le C/h.
\end{equation*}

Nontrapping resolvent estimates have application to Strichartz and local smoothing estimates \cite{bt07,mmt08}, resonance counting \cite{chr15}, and integrated local energy decay \cite{rota15}. For more about resolvent bounds under various dynamical assumptions, see chapter 6 from \cite{dyzw}, and the references therein. Note that, in our case, $\Phi(t)$ may be undefined, since $V$ may not be differentiable.

Let $\mathbf{1}_{\le 1}$ be the characteristic function of $\{ x \in \mathbb{R}^n : |x| \le 1 \}$, and define $\mathbf{1}_{\ge 1}$ similarly. The key to proving the Theorem is to establish the following global Carleman estimate. 
\begin{lemma}[Carleman estimate] \label{Carleman lemma}
Let $R_0 > 3$ so that $\supp V \subseteq B(0,R_0/2)$. There exist $K, C > 0$, $h_0 \in (0, 1]$, and $\varphi = \varphi_h \in C^2(0, \infty)$ depending on $E_{\text{\emph{min}}}$, $E_{\text{\emph{max}}}$, $\|V \|_\infty$, $R_0$, $n$, and $s$  such that
\begin{equation} \label{varphi h bound in lemma}
 \max \varphi =  K \log(h^{-1}), \qquad h \in (0, h_0],
\end{equation}
and
\begin{equation} \label{Carleman est}
\begin{split}
\Big\|\Big(\mathbf{1}_{\le 1}|x|^{1/2} + \mathbf{1}_{\ge 1} (1+|x|)^{-s} \Big) &e^{\varphi/h^{4/3}} v \Big\|^2_{L^2(\mathbb{R}^n)} \le \\
 \frac{C}{h^{10/3}}  \Big\|(1+|x|)^{s}&e^{\varphi/h^{4/3}} (P(h) - i\varepsilon)v \Big\|^2_{L^2(\mathbb{R}^n)} 
+  \frac{C\varepsilon}{h^{10/3}} \Big\| e^{\varphi/h^{4/3}}v \Big\|^2_{L^2(\mathbb{R}^n)},
\end{split}
\end{equation}
for all $E \in [E_\text{\emph{min}}, E_\text{\emph{max}} ]$, $\varepsilon > 0$, $h \in (0,h_0]$, and  $v \in C_0^\infty(\mathbb{R}^n)$.
\end{lemma}
The key properties of the Carleman weight $\varphi = \varphi_h$ are that $\partial_r \varphi$ is large on $\supp V$ and that $ \max \varphi =  K\log(h^{-1})$, where $K > 0$ depends on $E_{\text{min}}$, $\| V\|_\infty$, and $\supp V$, but not on $h$. We construct $\varphi$ to have these properties in Lemma \ref{phi lemma}.

To prove Lemma \ref{Carleman lemma}, we adapt the strategy appearing in \cite{cavo, da14, rota15, sh16}. The common starting point is a certain spherical energy functional $F: (0, \infty) \to \mathbb{R}$ that includes $\varphi$, see \eqref{defn F}. Typically, $F$ also  includes $V$. However, we intend to differentiate the product $wF$, where $w: (0, \infty) \to \mathbb{R}$ is a second weight function defined by \eqref{defn w}. Since we cannot differentiate $V$ in our case, we initially leave $V$ out of $F$, but add it back after differentiation. By doing so, we recover the terms needed to prove \eqref{Carleman est}, at the cost of introducing a remainder term that may be large on the support of $V$, which we must control. We control the remainder with two innovations that go beyond the techniques used in \cite{cavo, da14, rota15, sh16}. First we increase the $h$-dependence of the exponent in \eqref{Carleman est} to $h^{-4/3}$. This differs from the Carleman estimates in the previous works, which use a factor of the form  $e^{\varphi h^{-1}}$. Second, we require that $\partial_r \varphi \ge c$ on $\supp V$, where $c$ is chosen large enough to satisfy \eqref{lower bound c} and \eqref{lower bound tanh}.

The outline of the paper is as follows. In Section \ref{notation}, we construct the weights $w$  and $\varphi$ and prove their key properties. In Section \ref{Carleman section}, we prove the Carleman estimate. In Section \ref{proof of resolvent bound section}, we first glue two versions of the Carleman estimate togther to remove the loss at the origin. Then we prove the Theorem via a density argument. The density argument is straightforward and closely follows proofs in \cite{da14, sh16, dyzw}, but we recall it for the reader's convenience.

\textbf{Acknowledgements:} A previous version of this paper asserted only that $\max \varphi \le Kh^{-1/3}$, resulting in an larger $h^{-5/3}$ exponent on the right side of \eqref{exp bound}. Not until seeing the estimate \eqref{Klopp-Vogel} of Klopp and Vogel did the author realize that, without changing the construction, $\varphi'$ could be estimated more sharply outside the support of $V$ (see \eqref{remove power get log}). As a result, the author was able to improve the exponent in \eqref{exp bound} from $h^{-5/3}$  to $h^{-4/3}\log(h^{-1})$ where it currently stands. The author is grateful to Klopp and Vogel for helping to bring about this improvement.

The author is also thankful to Georgi Vodev, who provided the initial idea for this project, and to Kiril Datchev, Jeffrey Galkowski, and Michael Hitrik for helpful discussions and suggestions. Finally, the author was supported by the Bilsland Dissertation Fellowship from the Purdue University College of Science during the writing of this paper. 

\section{Notation and construction of the Carleman weight} \label{notation}
In this section, we establish notation, construct the weight functions $w$ and $\varphi$, and prove elementary estimates needed for the proof of Lemma \ref{Carleman lemma}.

Throughout the paper, we use prime notation to denote differentiation with respect to the radial variable $r \defeq |x|$, $x \in \mathbb{R}^n$. For instance, $u':=\partial_r u$. 

Put
\begin{equation*}
\delta \defeq 2s -1. 
\end{equation*}
Without loss of generality, we assume $0 < \delta < 1$. Fix $R_0 > 3$ large enough so that
\begin{equation*}
\supp V \subseteq B(0,R_0/2).
\end{equation*}

Next, choose $c > 1$ large enough so that 
\begin{gather}
c > \|V\|_\infty R_0 /4, \label{lower bound c} \\  
\sqrt{c} \tanh(\sqrt{c}/2) > \max\{\|V\|_\infty /4, 1\}. \label{lower bound tanh}
\end{gather}
Set
\begin{equation} \label{defn psi}
\psi = \psi_{h}(r) \defeq \begin{cases}  c & 0< r \le R_0, \\
\frac{B}{r^2} - \frac{h^{2/3}E_{\text{min}}}{4} & R_0 < r < R_1, \\
0 & r \ge R_1,
\end{cases}
\end{equation}
where we put 
\begin{equation} \label{defn B and R_1}
\begin{split}
 B& = B(h) \defeq \left( c + \frac{h^{2/3}E_{\text{min}}}{4} \right)R^2_0,  \\ R^2_1 = R^2_1&(h) \defeq \frac{4B}{h^{2/3}E_{\text{min}}} = \left(1 + \frac{4c}{h^{2/3}E_{\text{min}}}\right)R^2_0 ,
\end{split}
\end{equation}
so that $\psi$ is continuous. In Lemma \ref{phi lemma}, we will construct the Carleman weight $\varphi$ so that $(\varphi')^2$ is approximately equal to $\psi$  for $h$ small. From this relationship, we will deduce the properties of  $\varphi$ needed to prove the Carleman estimate.  

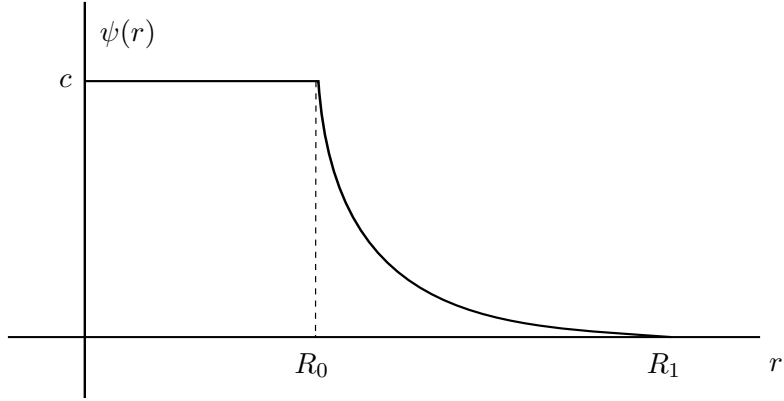
\begin{figure}
\centering
\labellist
\large
\pinlabel $R_0$  at 190 65
\pinlabel $c$  at 36 275
\pinlabel $R_1$ at 410 65
\pinlabel $\psi(r)$ at 75 310
\pinlabel $r$ at 480 65
\endlabellist
\includegraphics[width=10cm, height=6cm]{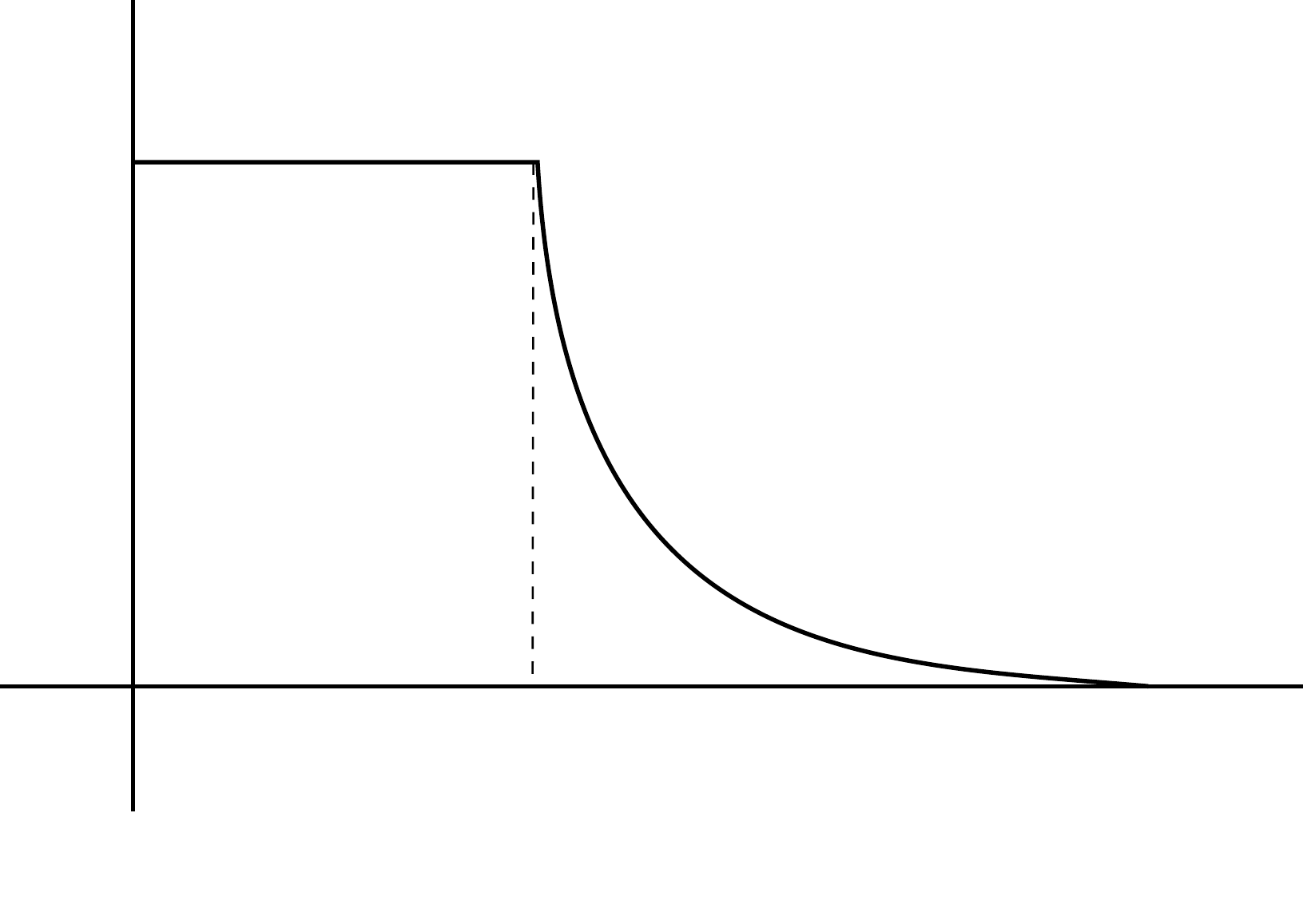}
 \caption{The graph of $\psi$.}
\end{figure}

\begin{figure}
\centering
\labellist
\large
\pinlabel $R_1$  at 233 65
\pinlabel $R^2_1$  at 12 254
\pinlabel $R^2_1+1$  at 0 296
\pinlabel $w(r)$ at 80 320
\pinlabel $r$ at 395 65
\endlabellist
\includegraphics[width=10cm, height=6cm]{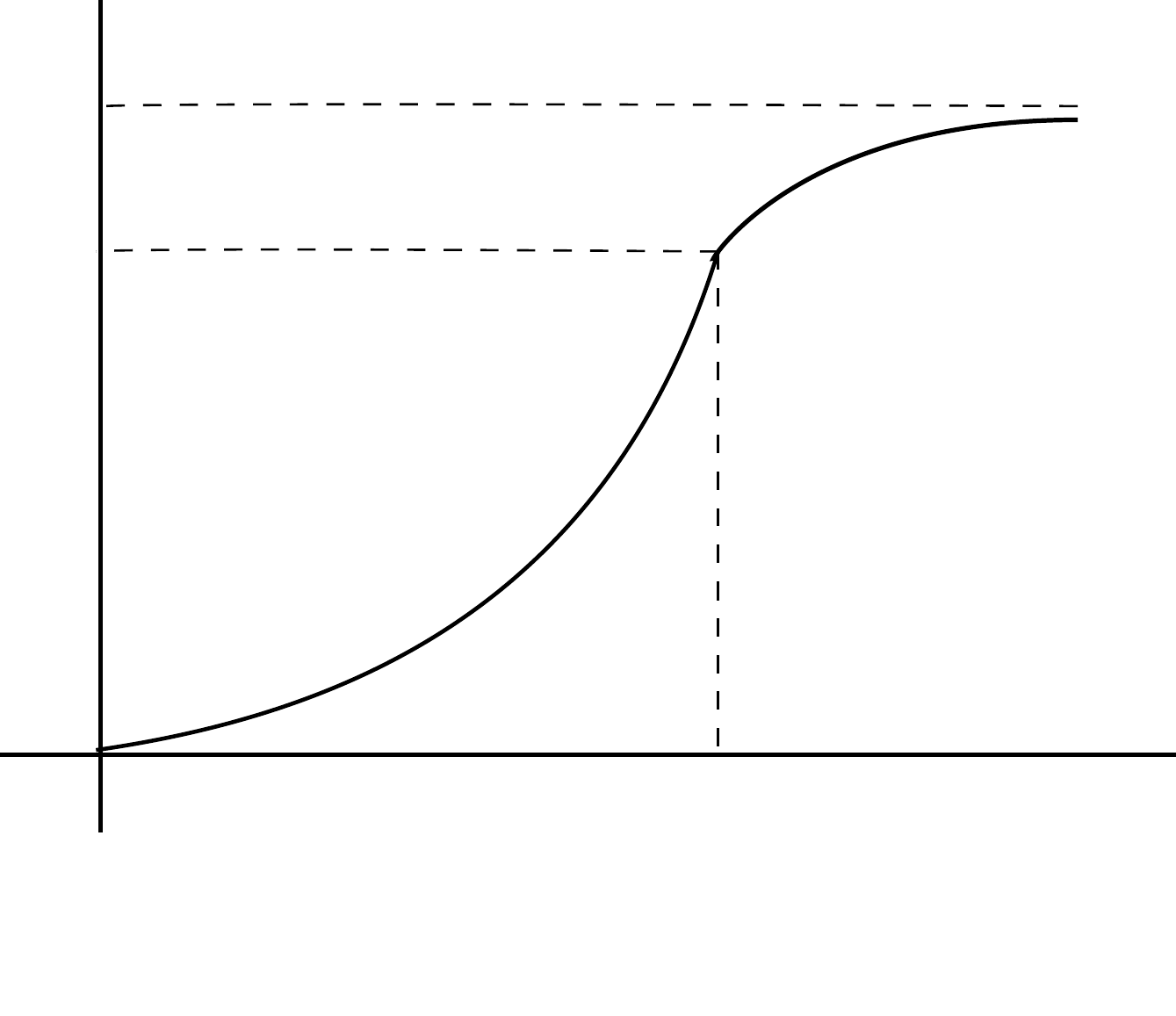}
 \caption{The graph of $w$.}
\end{figure}

To continue, define
\begin{equation} \label{defn w}
w = w_{h,\delta}(r) \defeq \begin{cases} r^2 & 0 < r < R_1, \\ R^2_1 + 1 -  (1 + (r - R_1))^{-\delta} & r \ge R_1.  \end{cases}
\end{equation}
According to \eqref{defn psi}, $\psi$ and $w$ satisfy the inequality 
\begin{equation} \label{w psi inequality}
h^{-2/3}(w\psi)' \ge  -\frac{E_{\text{min}}}{4}w', \qquad r>0, \text{ } r\neq R_0, \text{ } r \neq R_1. 
\end{equation}
We use \eqref{w psi inequality} in the proof of the Carleman estimate to ensure that a group of remainder terms is not too negative, see \eqref{minus two thirds est}. 

The next lemma proves elementary estimates involving $w$ and $w'$. We use them in the proof of Lemma \ref{Carleman lemma} to bring intermediate steps closer to \eqref{Carleman est}, note in particular \eqref{baseeqn2}.

\begin{lemma} \label{w w' inequalities}
Suppose $h \in (0,1]$. There exists $C > 1$ depending on $E_{\text{\emph{min}}}$, $R_0$, $c$, and $\delta$ so that for each $r \neq R_1$, it holds that
\begin{gather}
2wr^{-1} - w' \ge 0, \label{nonneg restrict w} \\
w(r) \le C h^{-2/3}  , \label{upper bound w}\\
w^2(r)/w'(r) \le C h^{-4/3} (1 + r)^{1 + \delta}, \label{upper bound w'} \\
w'(r) \ge C^{-1} \left(  \mathbf{1}_{\le 1}r +  \mathbf{1}_{\ge 1} (1 + r)^{-1 -\delta} \right). \label{lower bound w'}
\end{gather}
\end{lemma}
\begin{proof} When $r < R_1$, $2wr^{-1} - w' = 0$. If $r > R_1$, then 
\begin{equation*}
2wr^{-1} - w' = 2r^{-1}\left(R^2_1 + 1 - (1 + (r-R_1))^{-\delta}\right) - \delta (1 + (r-R_1))^{-1 - \delta}. 
\end{equation*}
So to finish proving \eqref{nonneg restrict w}, it is enough to show, 
\begin{equation*}
2R^2_1 \ge 2(1 + r - R_1)^{-\delta} + \delta r(1 + r-R_1)^{-1-\delta}, \qquad r > R_1.
\end{equation*}
Using $0 < \delta < 1$ and $R_1 > R_0 > 3$, we estimate,
\begin{equation*}
\begin{split}
2(1 + (r - R_1))^{-\delta} + \delta r\left(1 + r-R_1\right)^{-1-\delta} &\le  2 + \delta + \delta R_1 \\
& \le 2R^2_1.
\end{split}
\end{equation*}

To show \eqref{upper bound w}, simply note that 
\begin{equation*}
\begin{split}
w(r) &\le R^2_1 + 1 \\
&\le 2R^2_1 \\
& = 8 h^{-2/3} BE_{\text{min}}^{-1}.
\end{split}
\end{equation*}

For \eqref{upper bound w'}, when $0 < r \le R_1$,
\begin{equation*}
\begin{split}
w^2(r) /\left( w'(r) (1 + r)^{1 + \delta} \right)  &\le  2^{-1}r^2   \\
& \le 2^{-1} R^2_1\\
& = 2 h^{-2/3} BE_{\text{min}}^{-1}.
\end{split}
\end{equation*}
If $r \ge R_1$, then we use the bound $w(r) \le 2R^2_1$,
\begin{equation*}
\begin{split}
w^2(r)/\left( w'(r) (1 + r)^{1 + \delta} \right)  &= \delta^{-1}w^2(r) \left( \frac{1 +r - R_1}{1 +r } \right)^{1 + \delta} \\
&\le 4\delta^{-1}R^4_1  \\
& = 64 h^{-4/3} \delta^{-1} B^2 E_{\text{min}}^{-2}. 
\end{split}
\end{equation*}
As for \eqref{lower bound w'}, observe that when $1 < r \le R_1$,
\begin{equation*}
\begin{split}
w'(r)(1 + r)^{1+\delta} & = 2r(1 + r)^{1+\delta} \\
&\ge 2,
\end{split}
\end{equation*}
and when $ r > R_1 $,
\begin{equation*}
\begin{split}
 w'(r)(1 + r)^{1+\delta} & =  \delta \left( \frac{1 +r }{1 +r - R_1} \right)^{1 + \delta}.\\
&\ge \delta
\end{split}
\end{equation*}
\end{proof}

We now construct the Carleman weight $\varphi \in C^2(0, \infty)$ as a solution to an ODE with right hand side equal to $\psi$. The argument is modeled after Proposition 3.1 \cite{ddeh}.
\begin{lemma} \label{phi lemma}
Let $ h \in (0,1]$. There exists $\varphi = \varphi_h \in C^2(0, \infty)$ with the properties that
\begin{gather}
(\varphi')^2 - h^{4/3} \varphi'' = \psi, \qquad r >0, \label{solves ODE} \\
0 \le \varphi'(r) \le \sqrt{c}, \qquad r > 0, \label{bounds phi'} \\ 
0 \le \varphi'(r) \le Kr^{-1}, \qquad R_0 < r < R_1 \label{1/r bound phi'} \\
1 \le  \max \varphi = K\log(h^{-1}) \label{varphi h bound}, \\
\varphi'(r) \ge \sqrt{c} \tanh( \sqrt{c}/2), \qquad 0 < r < R_0/2, \label{lower bound phi' R_0/2} 
\end{gather}
where $K > 0$ depends on $\| V\|_\infty$, $R_0$ and $E_{\text{\emph{min}}}$ but not on $h$.
\end{lemma}
Once we construct $\varphi$ according to \eqref{solves ODE}, it holds that $\varphi' \approx \sqrt{\psi}$ for $h$ small, and so \eqref{bounds phi'} through \eqref{lower bound phi' R_0/2} follow naturally from the definition of $\psi$. 
\begin{proof}
To begin, consider the solution to the initial value problem
\begin{equation} \label{phi prime eqn}
y' = h^{-4/3}(y^2 - \psi), \qquad y(R_1) = 0.
\end{equation}
According to Theorem 1.2 in Chapter 1 of \cite{cole}, there exists an open interval $I$ containing $R_1$ and a solution $y \in C^1(I)$ to \eqref{phi prime eqn}. In fact, this solution is unique on $I$. For if $y_1$, $y_2$ are two solutions to \eqref{phi prime eqn}, then $\tilde{y} \defeq y_1 - y_2$ solves $\tilde{y}' = (y_1 + y_2) \tilde{y}$, $\tilde{y}(R_1) = 0$, and hence is identically zero. 

We take 
\begin{equation} \label{from y to phi}
\varphi(r) \defeq \int_0^r y(s) ds.
\end{equation}
Hence $\varphi$ satisfies \eqref{solves ODE}. We now analyze $y$ to establish \eqref{bounds phi'}, \eqref{varphi h bound} and \eqref{lower bound phi' R_0/2}.

First, we show that $y(r) = 0$ for $r \ge R_1$, $r \in I$, and therefore $y$ extends to be identically zero on $[R_1, \infty)$. Because $y(R_1) = 0$, there exists $\varepsilon \in (0,h^{4/3})$ so that $[R_1 , R_1 + \varepsilon) \subseteq I$ and $|y(r) |\le 1/2$ on  $[R_1 , R_1 + \varepsilon)$. Therefore, using \eqref{phi prime eqn}, we see that $|y'(r)| = h^{-4/3}|y(r)|^2 \le (4h^{4/3})^{-1}$ on $[R_1 , R_1 + \varepsilon)$. Hence
\begin{equation*}
\begin{split}
|y(r)| \le \int^r_{R_1} |y'(s)| ds 
&\le \frac{\varepsilon}{4h^{4/3}}\\
&\le \frac{1}{4}, \qquad r \in [R_1, R_1 + \varepsilon).
\end{split}
\end{equation*}
Applying $|y'(r)| = h^{-4/3}|y(r)|^2$ on $[R_1, R_1 + \varepsilon)$ another time, we then get $|y'(r)| \le (16h^{4/3})^{-1}$ and use it to show that $|y(r)| \le 16^{-1}$, $ r \in [R_1  , R_1 + \varepsilon)$. Continuing in this fashion, we see that $y(r) = 0$ for $r \in [R_1  , R_1 + \varepsilon)$. Therefore $y$ extends to be identically zero on $[R_1, \infty)$. 

Moving on, we now show that
\begin{equation} \label{bounds on y}
0 \le y \le \sqrt{\psi(R_0)} = \sqrt{c} 
\end{equation}
where it is defined on $(0,R_1]$. To show $y \ge 0$, assume for contradiction that there exists $0< r_0 < R_1$ with $y(r_0) < 0$. Then, because $y' = h^{-4/3}(y^2 - \psi) \le h^{-4/3}y^2$, we have $y'(r)/(y(r))^2 \le h^{-4/3}$, for $r$ near $r_0$. This implies 
\begin{equation} \label{bigger than zero}
\begin{split}
y(r_0)^{-1} - y(r)^{-1} & = \int^r_{r_0} \frac{y'(s)}{(y(s))^2} ds  \\
& \le \frac{r - r_0}{h^{4/3}}, \qquad \text{$r \ge r_0$, $r$ near $r_0$}.
\end{split}
\end{equation}
As $r$ approaches $\inf\{r \in [r_0, \infty): y(r) = 0 \} \le R_1$, \eqref{bigger than zero} must hold. But this is a contradiction because the left side becomes arbitrarily large, while the right side remains bounded. So $y(r) \ge 0$ where it is defined on $(0, R_1]$. 

To show $y \le \sqrt{\psi(R_0)}$, we compare $y$ to the solution of the initial value problem
\begin{equation*}
\begin{split}
z' &= h^{-4/3}(z^2 - \psi(R_0)) \\
& = h^{-4/3}(z^2 - c), \qquad z(R_1) = 0,
\end{split}
\end{equation*}
This solution exists for all $r > 0$ and is given by
\begin{equation*}
\begin{split}
z(r) &= \sqrt{c} \frac{1- \text{exp}\left(-2h^{-4/3}\sqrt{c} (R_1 - r) \right)}{1+ \text{exp}\left(-2h^{-4/3}\sqrt{c} (R_1 - r) \right) } \\
& = \sqrt{c} \tanh \left(h^{-4/3} \sqrt{c}(R_1 - r)\right).
\end{split}
\end{equation*}
Suppose for contradiction that there exists $r_0 < R_1$ such that $y(r_0) > z(r_0)$. Set $\zeta \defeq y -z$. Then
$\zeta'  \ge h^{-4/3}(y+z)\zeta$, $\zeta(r_0) >0$, and $\zeta(R_1) =0$. 

Put $r_1 \defeq \inf \{r \in (r_0,R_1]: \zeta(r) = 0\}$. By the mean value theorem, there exists $\tilde{r} \in (r_0, r_1)$ so that 
\begin{equation*}
\begin{split}
\zeta'(\tilde{r}) &=  \frac{\zeta(r_1) - \zeta(r_0)}{r_1 - r_0} \\
&= \frac{-\zeta(r_0)}{r_1 - r_0}\\
&< 0.
\end{split}
\end{equation*}
In addition, $\zeta(\tilde{r}) > 0$ by the definition of $r_1$. But this contradicts $\zeta'(\tilde{r}) \ge h^{-4/3}\zeta(\tilde{r})(y(\tilde{r}) + z(\tilde{r}))$ since $y + z \ge 0$ where $y$ is defined on $(0,R_1)$. 

So we have shown that  $0 \le y \le z \le \sqrt{c}$ where it is defined on $(0, R_1)$.It then follows by Theorem 1.3 in Chapter 2 of \cite{cole} that $y$ extends to all of $(0, R_1)$, where it obeys the same bounds.

We omit the proof of \eqref{1/r bound phi'}. However, we remark that one can show 
\begin{equation} \label{tildeB/r}
y \le \xi(r) \defeq \tilde{B}/r \qquad \text{on } (R_0, R_1), 
\end{equation}
where
\begin{equation*}
\tilde{B} \defeq \left(\sqrt{4B + h^{8/3}} - h^{4/3}\right)/2,
\end{equation*}
by first noting that $\xi$ solves 
\begin{equation}
\xi' = h^{-4/3}(\xi^2 - (B/r^2)), \quad \xi(R_1) = \tilde{B}/r,
\end{equation}
and then comparing $y$ and $\xi$ by the same method as in the preceding paragraph. 

Lastly, we show that 
\begin{equation} \label{y lower bound R_0/2}
y(r) \ge \sqrt{c}\tanh \left(\sqrt{c}/2 \right), \qquad r \in (0, R_0/2].
\end{equation}
To see this, let $\tilde{z}$ solve the initial value problem
\begin{equation*}
\tilde{z}' = h^{-4/3}(\tilde{z}^2 - \psi), \qquad \tilde{z}(R_0) = 0.
\end{equation*}
Then $\tilde{z}$ is given by
\begin{equation*}
\tilde{z}(r) = \sqrt{c}\tanh \left(h^{-4/3} \sqrt{c} \left(R_0 - r\right) \right).
\end{equation*}
Set $\tilde{\zeta} \defeq y - \tilde{z}$. To show \eqref{y lower bound R_0/2}, it is enough to see that $\tilde{\zeta} \ge 0$ on $(0, R_0)$, and we give an argument similar to the one in the preceding paragraph. For contradiction, suppose there exists $0 < r_2 \le R_0$ such that $\tilde{\zeta}(r_2) < 0$. Put $r_3 \defeq \inf\{r \in (r_2, R_0]: \tilde{\zeta}(r) = 0\}$. Such an $r_3$ exists because $\tilde{\zeta}(R_0) = y(R_0) \ge 0$. By the mean value theorem, there is some $r^* \in (r_2, r_3)$ so that $\tilde{\zeta}'(r^*) = -(r_3 - r_2)^{-1}\tilde{\zeta}(r_2) >0$, and furthermore $\tilde{\zeta}(r^*) < 0$ by the definition of $r_3$ . But also $\tilde{\zeta}'(r^*) = h^{-4/3}\tilde{\zeta}(r^*)(y(r^*) + \tilde{z}(r^*)) \le 0$, and so we have contradiction.

We now have enough properties of $y$ to finish the proof. With $\varphi$ defined by \eqref{from y to phi}, we observe that \eqref{bounds phi'} follows from \eqref{bounds on y}, and \eqref{lower bound phi' R_0/2} from \eqref{y lower bound R_0/2}. 

Lastly, we use \eqref{defn B and R_1}, \eqref{lower bound tanh}, \eqref{tildeB/r}, $R_0 > 3$, and $h \in (0,1]$ to see
\begin{gather}
 \max \varphi \ge \sqrt{c} \int_0^{R_0/2} \tanh \left(\sqrt{c}/2 \right)ds \ge 1, \nonumber \\
\max \varphi \le \int_0^{R_0} \sqrt{c} ds + \int_{R_0}^{R_1} \tilde{B}/s ds \le \sqrt{c} R_0 + \tilde{B}\log(R_1/R_0) \le K \log(h^{-1}) \label{remove power get log}, 
\end{gather}
where $K > 0$ depends on $\| V\|_\infty$, $R_0$ and $E_{\text{min}}$ but not on $h$. This shows  \eqref{varphi h bound} and completes the proof. 
\\
\end{proof}

\section{Proof of the Carleman estimate} \label{Carleman section}
In this section, we use the weight functions $w$ and $\varphi$ constructed in the previous section to prove Lemma \ref{Carleman lemma}. We make integral estimates using polar coordinates $(r, \theta) \in (0, \infty) \times \mathbb{S}^{n-1}$ on $\mathbb{R}^n$. As in the previous chapter, the starting point is a conveniently chosen conjugation 

\begin{equation*}
\begin{split}
P_{\varphi}   &\defeq e^{\varphi/h^{4/3}} r^{(n-1)/2} (P(h) - E - i\varepsilon) r^{-(n-1)/2}  e^{-\varphi/h^{4/3}} \\
 &= -h^2 \partial_r^2 + 2 h^{2/3} \varphi' \partial_r + \Lambda + \rho  +V -h^{-2/3}\psi - E - i \varepsilon, 
 \end{split}
\end{equation*}
where 
\begin{equation*}
0 \le \Lambda = \Lambda_h(r) \defeq  -h^2r^{-2}\Delta_{\mathbb{S}^{n-1}}, \qquad \rho = \rho_h(r) \defeq h^2(2r)^{-2}(n-1)(n-3).
\end{equation*}
To prove the Carleman estimate, we need another simple estimate, this time involving involving $w$, $w'$ and $\rho$.
\begin{lemma} \label{control effective potential}
There exists $ h_0 \in (0, 1]$ depending on $E_{\text{\emph{min}}}$ and $n$ so that 
\begin{equation} \label{effective potential estimate}
\left(2w(r)r^{-1} - w'(r)\right)\rho(r) \ge -\frac{ E_{\text{\emph{min}}}}{4}w'(r), 
\end{equation}
for all $E \ge E_{\text{\emph{min}}}$, $r \neq R_1$, and $h \in (0,h_0]$.
\begin{proof}
If $r < R_1$, then $2wr^{-1} - w' = 0$ and \eqref{effective potential estimate} follows immediately. On the other hand, if $r > R_1$, we use $R_1 > 3$ to see that  
\begin{equation*}
\begin{split}
\left(2w(r)r^{-1} - w'(r)\right)\rho(r) &\ge -h^2(2r)^{-2}|n-1||n-3| w'(r) \\
& \ge  -h^2 |n-1||n-3| w'(r)/36.
\end{split}
\end{equation*}
So we obtain \eqref{effective potential estimate} for $r > R_1$ by taking $h_0$ sufficiently small. \\
\end{proof}

\end{lemma}
\begin{proof}[Proof of Lemma \ref{Carleman lemma}]
Let $\int_{r, \theta}$ denote the integral over $(0, \infty) \times \mathbb{S}^{n-1}$ with respect to $dr d\theta$, where $d \theta$ is the usual surface measure on $\mathbb{S}^{n-1}$. 

To show \eqref{Carleman est}, it suffices to prove that
\begin{equation} \label{polar Carleman}
\begin{aligned}
\int_{r, \theta} ( \mathbf{1}_{\le 1} r + &\mathbf{1}_{\ge 1} (1 + r)^{-1 - \delta} )|u|^2 \le \\ &\frac{C}{h^{10/3}} \left( \int_{r, \theta}(1 + r)^{1 + \delta} |P_\varphi u|^2 + \varepsilon
\int_{r, \theta} |u|^2 \right), \quad u \in r^{(n-1)/2}e^{\varphi/h^{4/3}}C_0^\infty(\mathbb{R}^n). 
\end{aligned}
\end{equation}
Without loss of generality, we may assume $\varepsilon \le h^{10/3}$. To show \eqref{polar Carleman}, we proceed in the spirit of the previous chapter and of \cite{cavo, da14, rota15} and define the functional $F$ by
\begin{equation} \label{defn F}
F(r) \defeq \| hu'\|^2_S - \langle (\Lambda + \rho -h^{-2/3}\psi- E) u,u  \rangle_S, \qquad r > 0,
\end{equation} 
where $\| \cdot \|_S$ and $\langle \cdot , \cdot \rangle_S$ denote the norm and inner product on $\mathbb{S}^{n-1}$, respectively.

We compute the derivative of $F$, which exists for all $r \neq R_0$, $r \neq R_1$, 
\begin{equation*}
\begin{split}
F'(r) &=  2 \real \langle h^2 u'', u' \rangle_S - 2 \real \langle (\Lambda + \rho -h^{-2/3}\psi - E)u, u' \rangle_S \\
&+ 2r^{-1} \langle (\Lambda + \rho) u , u \rangle  + \langle h^{-2/3}\psi' u ,u \rangle_S \nonumber. 
\end{split}
\end{equation*}
Next, we calculate, for $r \neq R_0,$ $r \neq R_1$,
\begin{equation*}
\begin{split}
wF' + w'F &= 2w \real \langle h^2 u'', u' \rangle_S - 2w \real \langle (\Lambda + \rho -h^{-2/3}\psi - E) u, u' \rangle_S \\
&+ 2wr^{-1} \langle (\Lambda + \rho) u,u \rangle_S + h^{-2/3} w \psi' \| u \|^2_S\\
&+ w'\|hu' \|^2_S - w'  \langle (\Lambda + \rho) u,u \rangle_S +  w'\langle (h^{-2/3}\psi + E)u,u \rangle_S \\
&= -2w \real \langle P_\varphi u, u' \rangle_S  + 2w\varepsilon \imag \langle u, u' \rangle_S \\
&+ h^2w' \| u' \|^2_S + (2wr^{-1} - w')\langle (\Lambda + \rho) u,u \rangle_S \\
&+ Ew'\|u\|^2_S + 4h^{2/3}w \varphi'  \| u' \|^2_S   +h^{-2/3}(w \psi)' \|u \|^2_S+   2 w \real \langle Vu ,u' \rangle_S.\\
\end{split}
\end{equation*}
Note that we have have added and subtracted $2 w \real \langle Vu ,u' \rangle_S$, $4h^{2/3}w\varphi'\|u\|^2_S$, and $2w\varepsilon \imag \langle u,u' \rangle_S$ in order to recover $P_\varphi u$ in line four. Using $w' > 0$, $2wr^{-1} - w' \ge 0$, $\Lambda \ge 0$ and $-2\real\langle a,b \rangle + \| b\|^2 \ge -\|a \|^2$, we estimate, for $r \neq R_0,$ $r \neq R_1$,
 \begin{equation} \label{intermediate w'F + wF'}
 \begin{split}
wF' + w'F &\ge -\frac{w^2}{h^2w'} \|P_\varphi u \|^2_S + 2 w \varepsilon \imag \langle u, u' \rangle_S  \\
&+ Ew' \| u \|^2_S + (2wr^{-1} - w')\rho \| u  \|^2_S \\
&+ 4h^{2/3}w \varphi' \|u'\|^2_S  + h^{-2/3}(w \psi)' \| u \|^2_S +  2 w \real \langle Vu ,u' \rangle_S.\\
\end{split}
\end{equation} 

To continue, let $\mathbf{1}_{B(0,R_0/2)}$ denote the characteristic function of $B(0,R_0/2)$. We bound $2 w \real \langle Vu ,u' \rangle_S$ from below by
\begin{equation*}
\begin{split}
2 w \real \langle Vu, u' \rangle_S & \ge - 2w(r) \int_{\theta} |V(r, \theta)u(r, \theta)u'(r, \theta)| d\theta\\ 
& \ge - \gamma\| V \|_\infty \mathbf{1}_{B(0,R_0/2)}(r) w(r)   \| u'(r, \theta) \|^2_S \\
& -\gamma^{-1} \| V \|_\infty \mathbf{1}_{B(0,R_0/2)}(r)  w(r) \|u(r, \theta) \|^2_S, \qquad \gamma > 0.
\end{split}
\end{equation*}
Plugging this lower bound into \eqref{intermediate w'F + wF'}, we get for  $r \neq R_0, r \neq R_1$.
\begin{equation} \label{w'F + wF'}
\begin{split} 
wF' + w' F &\ge -\frac{w^2}{h^2w'} \|P_\varphi u \|^2_S + 2 w \varepsilon \imag \langle u, u' \rangle_S  \\
& + \left(4h^{2/3} \varphi' - \gamma\| V \|_\infty \mathbf{1}_{B(0,R_0/2)} \right)w \| u' \|^2_S \\
& +\left( Ew'+ (2wr^{-1} - w')\rho + h^{-2/3}(w \psi)' - \gamma^{-1} \| V \|_\infty \mathbf{1}_{B(0,R_0/2)} w\right) \| u \|^2_S. \\
\end{split}
\end{equation}
Now, fix $\gamma = h^{2/3}$ (the author is grateful to Jeff Galkowski for the suggestion to use an $h$-dependent $\gamma$). Then, use $\psi = c$ on $(0,R_0]$ along with \eqref{lower bound c} to get
\begin{equation*}
\begin{split}
(w \psi)' - \| V \|_\infty \mathbf{1}_{B(0,R_0/2)}w  & \ge r\left(2c  -\| V \|_\infty R_0/2 \right) \\
& \ge 0, \qquad r \in (0, R_0/2].
\end{split}
\end{equation*}
Combining this with \eqref{w psi inequality} and \eqref{effective potential estimate}, we have 
\begin{equation} \label{minus two thirds est}
\left(Ew' + (2wr^{-1} - w')\rho + h^{-2/3}(w \psi)' - \gamma^{-1} \| V \|_\infty \mathbf{1}_{B(0,R_0/2)}w\right) \|u\|_S^2 \ge \frac{E_{\text{min}}}{2}w'\|u\|^2_S.
\end{equation}
for all $r > 0$, $r \neq R_0$, $r \neq R_1$, and all $h \in (0, h_0]$, where $h_0$ is as given in Lemma \ref{control effective potential}.

On the other hand, according to \eqref{lower bound tanh}, \eqref{bounds phi'}, and \eqref{lower bound phi' R_0/2}, we have
\begin{equation*}
\begin{split}
4 \varphi' - \| V \|_\infty \mathbf{1}_{B(0,R_0/2)} \ge 0 , \qquad r >0.
\end{split}
\end{equation*}

Updating \eqref{w'F + wF'} with these lower bounds, we get
\begin{equation} \label{final w'F + wF'}
\begin{split}
w F' + w' F &\ge -\frac{w^2}{h^2w'} \|P_\varphi \|^2_S + 2 w \varepsilon \imag \langle u, u' \rangle_S \\
&+ \frac{E_{\text{min}}}{2}w' \| u\|_S^2, \qquad r \neq R_0, \text{ } R_1. 
\end{split} 
\end{equation}
Next, we apply Fatou's lemma, along with the fundamental theorem of calculus to get
\begin{equation} \label{Fatou}
\int_0^\infty (w(r)F(r))'  \le -\liminf_{r \to 0} w(r)F(r) = 0.
\end{equation} 
Integrating \eqref{final w'F + wF'} with respect to $dr$ and using \eqref{Fatou}, we arrive at 
\begin{equation} \label{baseeqn}
\frac{E_{\text{min}}}{2}\int_{r, \theta} w'|u|^2 \le \frac{1}{h^2} \int_{r, \theta} \frac{w^2}{w'} |P_{\varphi} u|^2 + 2 \varepsilon \int_{r, \theta}w |uu'|.
\end{equation}
Combining \eqref{baseeqn} with, \eqref{upper bound w'} and \eqref{lower bound w'} gives for $h \in (0,h_0]$
\begin{equation} \label{baseeqn2} 
 \int_{r, \theta} \left( \mathbf{1}_{\le 1} r + \mathbf{1}_{\le 1} (1 + r)^{-1 - \delta}\right)|u|^2 \le \frac{C}{h^{10/3}} \int_{r, \theta} (1 + r)^{1 + \delta}|P_{\varphi} u|^2  + 2 \varepsilon \int_{r, \theta} w|uu'|,
\end{equation}
where $C>1$ is a constant that depends on $E_{\text{min}}$, $R_0$, $n$, $c$ and $\delta$, but is independent of $h$ and $u$. We will reuse $C$ is the ensuing estimates, but its precise value will change from line to line.

We focus on the last term in \eqref{baseeqn2}. Our goal is to show 
\begin{equation} \label{bound uu'}
2\int_{r, \theta} w|uu'| \le \frac{C}{h^2} \left( \int_{r, \theta} w^2|P_{\varphi} u|^2 +  \int_{r, \theta} \left(1 + w^2 + \rho w^2\right) |u|^2\right), \qquad h \in (0,h_0]. 
\end{equation}
 If we have shown \eqref{bound uu'}, we can substitute it into \eqref{baseeqn2} and use \eqref{upper bound w} along with
 \begin{equation*}
 |\rho w^2| \le Ch^{2/3}, \qquad r > 0
 \end{equation*}
to get
\begin{equation*} 
\begin{split}
\int_{r, \theta} \left( \mathbf{1}_{\le 1} r + \mathbf{1}_{\le 1} (1 + r)^{-1 - \delta}\right)|u|^2    &\le \frac{C}{h^{10/3}} \int_{r, \theta} (1 + r)^{1 + \delta} |P_{\varphi} u|^2 + \frac{C\varepsilon}{h^{10/3}} \int_{r, \theta} |P_{\varphi} u|^2  \\
&+ \frac{C \varepsilon}{h^{10/3}} \int_{r, \theta} |u|^2, \qquad h \in (0, h_0].
\end{split}
\end{equation*}
Using $\varepsilon \le h^{10/3}$  then gives \eqref{polar Carleman}.

To show \eqref{bound uu'}, we first write
\begin{equation} \label{peter paul uu'}
2\int_{r, \theta} w|uu'| \le \frac{1}{h^2} \int_{r, \theta} |u|^2 +  \int_{r, \theta} w^2|hu'|^2. 
\end{equation}
We will now show that
\begin{equation} \label{bound on w^2|hu'|^2}
 \int_{r, \theta} w^2|hu'|^2 \le C  \int_{r, \theta} w^2|P_{\varphi} u|^2 +  \frac{C}{h^{2/3}} \int_{r, \theta} (w^2 + |\rho w^2|) |u|^2 , \qquad h \in (0,h_0],
\end{equation}
which will complete the proof of the Lemma. To show \eqref{bound on w^2|hu'|^2}, we first integrate by parts,
\begin{equation*}
\begin{split} 
 \int_{r, \theta} w^2|hu'|^2 &=  \real \left(  \int_{r, \theta} w^2 \bar{u}(-h^2u'') -2h^2w w' \bar{u}u'  \right),
\end{split}
\end{equation*}
and then estimate,
\begin{equation} \label{easy by parts estimate}
 \real \int_{r, \theta} -2h^2w w' \bar{u}u' \le \frac{h^2}{\eta_1} \int_{r, \theta} (w')^2 |u|^2+ \eta_1 \int_{r, \theta} w^2 |h u'|^2, \qquad \eta_1 > 0,
\end{equation}
 \begin{equation} \label{hard by parts estimate}
 \begin{split} 
  \real  \int_{r, \theta} w^2 \bar{u}&(-h^2u'') \\
  &=  \real \int_{r, \theta}  w^2 \bar{u} (P_{\varphi} - 2h^{2/3}\varphi' \partial_r - \Lambda - \rho - V+ h^{-2/3}\psi + E + i \varepsilon )u 
\\& \le  \int_{r, \theta} w^2|P_{\varphi}u||u| + 2 \int_{r, \theta} w^2 \varphi' |h^{2/3}u'||u| \\
&+  \int_{r, \theta}  w^2|E-\rho- V + h^{-2/3}\psi||u|^2
\\& \le  \frac{1}{2}  \int_{r, \theta} w^2 |P_{\varphi}u|^2 + \eta_2 \sqrt{c} \int_{r, \theta} w^2|hu'|^2 +  \int_{r, \theta} |\rho w^2| |u|^2 \\&+  \left( \frac{\sqrt{c}}{h^{2/3}\eta_2} + E_{\text{max}} + \| V\|_\infty + \frac{c}{h^{2/3}} + \frac{1}{2} \right)  \int_{r, \theta} w^2 |u|^2, \qquad \eta_2 > 0.
\end{split}
\end{equation}
Now, take $\eta_1 = 1/4$,  $ \eta_2 = 1/(4\sqrt{c})$, and bound $\int_{r, \theta} w^2|hu'|^2$ from above in \eqref{peter paul uu'} using \eqref{easy by parts estimate} and \eqref{hard by parts estimate}. We get, for $h \in (0, h_0]$,
\begin{equation*}
 \int_{r, \theta} w^2|hu'|^2 \le C \int_{r, \theta}  w^2|P_{\varphi}u|^2 + \frac{C}{h^{2/3}} \int_{r, \theta} (w^2 + \rho w^2) |u|^2 + \frac{1}{2} \int_{r, \theta}  w^2|hu'|^2.
\end{equation*}
Subtracting the last term to the left side and multiplying through by 2, we arrive at \eqref{bound on w^2|hu'|^2}.\\
\end{proof}

\section{Proof of the theorem} \label{proof of resolvent bound section}

In this final section, we use Lemma \ref{Carleman lemma} to prove the Theorem. We condense notation by setting $L^2 = L^2(\mathbb{R}^n)$, $H^1 = H^1(\mathbb{R}^n)$, $H^2 = H^2(\mathbb{R}^n)$, $C^\infty_0 = C^\infty_0(\mathbb{R}^n)$, and by renaming the weight appearing on the left side of \eqref{Carleman est},
\begin{equation*}
b(r) \defeq \mathbf{1}_{\le 1} r^{1/2} + \mathbf{1}_{\ge 1}(1 + r)^{-s}.  
\end{equation*} 
We also employ of a smooth version of the weight $(1 + r)^s$, which we denote by $m$, 
\begin{equation*}
m = m_\delta(r) \defeq (1 + r^2)^{(1 + \delta)/4}.
\end{equation*}

Before giving the main argument, we make two reductions. First, since
\begin{equation*}
(1 + r)^s/\sqrt{2} \le m(r) \le (1 + r)^s, \qquad r > 0, 
\end{equation*}
to prove the Theorem it suffices to show \eqref{exp bound} holds except with each instance of $(1 + |x|)^{-s}$ replaced by $m^{-1}$. Second, to obtain the desired $L^2 \to H^2$ bound, we merely need to show
\begin{equation} \label{can reduce to L2 to L2}
\begin{split}
  \| m^{-1} (P(h) - E &-  i\varepsilon)^{-1} m^{-1} \|_{L^2 \to L^2}  \le e^{Ch^{-4/3} \log(h^{-1})}, \\
  E \in [E_{\text{min}}, &E_{\text{max}}], \text{ } 0 < \varepsilon < 1, \text{ } h \in (0,h_0].
  \end{split}
\end{equation}
The argument for making this reduction is standard, but we give it now for the sake of completeness.  

Throughout the followings estimates, and later in the proof of the Theorem, $C > 1$ denotes a constant that is independent of $h$, but may depend on $E_{\text{min}}$, $E_{\text{max}}$, $\supp V$, $\|V \|_\infty$, $n$, and $s$, It's precise value will change from the line to line. 

For each $f \in H^2$, it holds that
\begin{equation*}
\|f \|_{H^2} \le C \left( \| f \|_{L^2} + \|\Delta f \|_{L^2} \right).
\end{equation*}
Therefore to show \eqref{exp bound}, we only need that  
\begin{equation} \label{commute Laplacian}
\begin{split}
\| \Delta m^{-1} (P(h) - &E- i\varepsilon)^{-1} m^{-1} f \|_{L^2} \le e^{Ch^{-4/3} \log(h^{-1})}\|f\|_{L^2}, \\
 E \in [E_{\text{min}}, &E_{\text{max}}], \text{ } 0 < \varepsilon < 1, \text{ } h \in (0,h_0], \text{ } f \in L^2.
\end{split}
\end{equation}

If  $[\Delta, m^{-1}]$ denotes the commutator of $\Delta$ and $m^{-1}$, then a simple calculation shows 
\begin{equation*}
[\Delta, m^{-1}]mf =  m^{-1} (\Delta m) f + 2 m^{-1} \nabla m \cdot \nabla f, \qquad f \in H^1,
\end{equation*}
which is a bounded map $H^1 \to L^2$. 
Using \eqref{can reduce to L2 to L2} along with
\begin{equation*}
\|\nabla f \|_{L^2} \le C\gamma \|f\|_{L^2} + \gamma^{-1} \|\Delta f\|_{L^2}, \qquad \gamma > 0, \text{ } f \in L^2,
\end{equation*}
and
\begin{equation*}
\Delta (P(h) - E - i \varepsilon)^{-1} = h^{-2} \left(V - E - i\varepsilon \right) (P(h) - E - i\varepsilon)^{-1} -h^{-2},
\end{equation*}
we have for $E \in [E_{\text{min}}, E_{\text{max}}]$ and $h$ small enough, 
\begin{equation*}
\begin{split}
\|\Delta m^{-1} (P(h) - E -i\varepsilon)^{-1} m^{-1}f\|_{L^2} & \le \| [\Delta, m^{-1}]  mm^{-1}(P(h) - E -i\varepsilon)^{-1} m^{-1}f \|_{L^2}  \\
&+ \| m^{-1} \Delta (P(h) - E -i\varepsilon)^{-1} m^{-1} f \|_{L^2} \\
& \le  C\|m^{-1}  (P(h) - E -i\varepsilon)^{-1} m^{-1} f\|_{H^1} \\
&+   Ch^{-2} e^{Ch^{-4/3} \log(h^{-1})} \| f \|_{L^2} \\
& \le C(1 + \gamma) e^{Ch^{-4/3} \log(h^{-1})} \| f \|_{L^2} \\
&+ C\gamma^{-1}\|\Delta m^{-1}(P(h) - E -i\varepsilon)^{-1} m^{-1} f\|_{L^2} \\
&+   e^{Ch^{-4/3} \log(h^{-1})} \| f \|_{L^2}.
\end{split}
\end{equation*}
If we set $\gamma = 2C$, we can absorb the term in line six on the right side into the left side, and then multiply through by $2$. This establishes \eqref{commute Laplacian}.
\begin{proof}[Proof of the Theorem]

Let $\tilde{R}_0 > 3$ be large enough so that $\supp V \subseteq B(0, \tilde{R}_0/4)$. Pick $ x_0 \in \mathbb{R}^n$ with $1/2 < |x_0| < 3/4$, which implies
\begin{equation*}
\supp V_0(\cdot + x_0) \subseteq B(0, \tilde{R}_0/2).
\end{equation*}
We shift coordinates, apply \eqref{Carleman est} to the operator $P_0 = P_0(h) \defeq -h^2 \Delta + V( \cdot + x_0) -E$ in place of $P$, and then shift back:
\begin{equation*}
\begin{split}
\left\| b(| \cdot - x_0|) e^{\varphi(| \cdot - x_0|)h^{-4/3}} v \right\|^2_{L^2}   &=\left\| b e^{\varphi h^{-4/3}} v ( \cdot + x_0) \right\|^2_{L^2}  \\& \le \frac{C}{h^{10/3}} \left\| me^{\varphi h^{-4/3}}(P_0 - i \varepsilon) v( \cdot + x_0) \right\|^2_{L^2} \\  &+ \frac{C\varepsilon}{h^{10/3}} \left\| e^{\varphi h^{-4/3}} v( \cdot + x_0) \right\|^2_{L^2}
\\&=  \frac{C}{h^{10/3}} \left\| m(| \cdot - x_0|)e^{\varphi(| \cdot - x_0|) h^{-4/3}}(P - i \varepsilon) v \right\|^2_{L^2}
\\&+ \frac{C\varepsilon}{h^{10/3}} \left\| e^{\varphi(| \cdot - x_0|) h^{-4/3}} v \right\|^2_{L^2}, \qquad h \in (0, h_0].
\end{split}
\end{equation*}
Summarizing in a single inequality, we have 
\begin{equation} \label{shifted estimate}
\begin{split}
\left\| b(|\cdot - x_0 |) e^{\varphi(| \cdot - x_0|) h^{-4/3}} v \right\|_{L^2} &\le  \frac{C}{h^{10/3}} \left\| m(| \cdot - x_0|)e^{\varphi(| \cdot - x_0|) h^{-4/3}}(P - i \varepsilon) v \right\|_{L^2} \\
&+ \frac{C\varepsilon}{h^{10/3}} \left\| e^{\varphi(| \cdot - x_0|) h^{-4/3}} v \right\|_{L^2}, \qquad h \in (0, h_0].
\end{split}
\end{equation}

Set $C_\varphi = C_\varphi(h) \defeq 2 \max \varphi$. Recall that by \eqref{varphi h bound},
\begin{equation} \label{varphi h bound reminder}
1 \le C_\varphi \le K \log(h^{-1}),
\end{equation}
 for $K >0$ depending on $\tilde{R}_0$, $\|V\|_\infty$, and $E_{\text{min}}$, but not on $h$.   Multiply \eqref{Carleman est} and \eqref{shifted estimate} through by $e^{-C_\varphi h^{-4/3}}$  to obtain for $h \in (0, h_0]$,
\begin{equation} \label{shift}
e^{-C_\varphi h^{-4/3}} \|b v \|^2_{L^2}  \le \frac{C}{h^{10/3}} \|m(P - i \varepsilon)v \|^2_{L^2} + \frac{C \varepsilon}{h^{10/3}} \|v\|_{L^2}^2, 
\end{equation}
\begin{equation} \label{no shift}
e^{-C_\varphi h^{-4/3}} \|b(| \cdot - x_0|) v \|^2_{L^2}  \le \frac{C}{h^{10/3}} \|m(|\cdot - x_0|)(P - i \varepsilon )v\|_{L^2}^2 + \frac{C \varepsilon}{h^{10/3}} \|v\|_{L^2}^2. 
\end{equation}
It is straightforward to show that 
\begin{equation} \label{remove shift}
4^{-1}m^{-2} \le b^2 + b^2(|\cdot - x_0|), \qquad  m^2 + m^2(|\cdot - x_0|)) \le 17m^2,
\end{equation}
 We add \eqref{no shift} and \eqref{shift} and apply \eqref{remove shift} to arrive at
\begin{equation*} 
e^{-C_\varphi h^{-4/3}}\|m^{-1}  v \|_{L^2}^2 \le \frac{C}{h^{10/3}} \|m(P - i \varepsilon)v \|_{L^2}^2 + \frac{C \varepsilon}{h^{10/3}} \|v\|_{L^2}^2.
\end{equation*}

For any $ \eta >0$, 
\begin{equation*}
\begin{split}
2\varepsilon \| v \|^2_{L^2} &= -2 \imag\langle (P - i\varepsilon)v, v \rangle_{L^2} 
\\& \le  \eta^{-1}\|m(P- i \varepsilon)v \|^2_{L^2} 
+ \eta\|m^{-1} v\|^2_{L^2}.  
\end{split}
\end{equation*}
Setting $\eta = h^{10/3} (2C)^{-1} e^{-C_\varphi h^{-4/3}} $ and applying \eqref{varphi h bound reminder}, we estimate $\varepsilon \| v \|^2_{L^2}$  from above and find that
\begin{equation} \label{penult}
 \|m^{-1} v \|_{L^2}^2 \le e^{C h^{-4/3}\log(h^{-1})} \|m(P - i \varepsilon)v \|_{L^2}^2, \qquad h \in (0,h_0].
\end{equation}

The final task is to use \eqref{penult} to show that for any $f \in L^2$,
\begin{equation} \label{ult}
\|m^{-1}(P - i\varepsilon)^{-1} m^{-1} f \|^2_{L^2}  \le e^{C h^{-4/3}\log(h^{-1})} \|f \|_{L^2}^2, \qquad h \in (0,h_0].
\end{equation}
from which \eqref{exp bound} follows. To establish \eqref{ult}, we prove a simple Sobolev space estimate  and then apply a density argument which relies on \eqref{penult}. 

In what follows, we use $a \lesssim b$ to denote $a \le C_{\varepsilon,h}b$ for $C_{\varepsilon,h}$ depending on $\varepsilon$ and $h$, but not on $v \in H^2$. The commutator $[P,m] = -h^2 \Delta m + 2 h^2 \nabla m \cdot \nabla : H^2 \to L^2$ is bounded. So for $v \in H^2$ such that $mv \in H^2$, we have

\begin{equation*}
 \begin{split}
\|m(P - i \varepsilon)v\|_{L^2} & \le \|(P - i \varepsilon)m v \|_{L^2} +  \|[P,m]v \|_{L^2}
\\& \lesssim \| mv \|_{H^2} + \|v \|_{H^2}
\\& \lesssim  \| mv \|_{H^2}.
\end{split} 
\end{equation*}
Thus we have shown 
\begin{equation} \label{Ceph}
\|m(P-i\varepsilon)v\|_{L^2} \le C_{\varepsilon,h} \|mv\|_{H^2}, \qquad \text{$v \in H^2$ such that $mv \in H^2$}.
\end{equation}

For fixed $f \in L^2$, the function $m(P-i\varepsilon)^{-1}m^{-1} f \in H^2$ because 
\begin{equation*}
\begin{split}
m(P-i\varepsilon)^{-1}m^{-1} f &= (P - i\varepsilon)^{-1} f + [m, (P-i\varepsilon)^{-1}] m^{-1}f  
\\& =  (P - i\varepsilon)^{-1} f + (P -i\varepsilon)^{-1} [P,m]  (P -i\varepsilon)^{-1} m^{-1}f.
\end{split}
\end{equation*}
Now, choose a sequence $v_k \in C_0^\infty$ such that $ v_k \to  m(P-i\varepsilon)^{-1}m^{-1} f$ in $H^2$. Define $\tilde{v}_k \defeq m^{-1}v_k$. Then, as $k \to \infty$
\begin{equation*}
\| m^{-1} \tilde{v}_k - m^{-1} (P- i \varepsilon)^{-1}m^{-1}f \|_{L^2} \le \| v_k - m (P- i \varepsilon)^{-1}m^{-1}f \|_{H^2} \to 0.
\end{equation*}
Also, applying \eqref{Ceph}
\begin{equation*}
\|m(P- i \varepsilon)\tilde v_k - f\|_{L^2} \lesssim \|v_k - m (P- i \varepsilon)^{-1} m^{-1} f \|_{H^2} \to 0.
\end{equation*} 
We then achieve \eqref{ult} by replacing $v$ by $\tilde{v}_k$ in \eqref{penult} and sending $k \to \infty$.\\
\end{proof}

\end{document}